\documentclass[10pt]{amsart}

\usepackage{amsmath,amssymb,amsthm,xspace,verbatim}

\theoremstyle{plain}
\newtheorem{thm}{Theorem}
\newtheorem{lem}[thm]{Lemma}

\newtheorem{cor}[thm]{Corollary}

\theoremstyle{definition}
\newtheorem*{dfn}{Definition}

\newcommand{\diag}{\operatorname{diag}}
\newcommand{\krank}{{K-rank\ }}

\begin{document}

\title{A concise proof of Kruskal's theorem on tensor decomposition}
\date{\today}

\author{John A. Rhodes}
\address{Department of Mathematics and Statistics\\University of Alaska Fairbanks\\PO Box 756660\\ Fairbanks, AK 99775}
\email{j.rhodes@uaf.edu}
\thanks{The author thanks the National Science Foundation, for support from NSF grant DMS
0714830, and the Statistical and Applied Mathematical Sciences Institute, for support during a residency in which this work had its origins.}

\maketitle

\begin{abstract}
A theorem of J.~Kruskal from 1977, motivated by a latent-class statistical model, established that under certain explicit conditions the expression of a 3-dimensional tensor as the sum of rank-1 tensors is essentially unique. We give a new proof of this fundamental result, which is substantially shorter than both the original one and recent versions along the original lines.
\end{abstract}

\section{Introduction} 

In \cite{MR0444690}, J.~Kruskal proved that, under certain explicit conditions, the expression of a 3-dimensional tensor (\emph{i.e.}, a 3-way array) of rank $r$ as a sum of $r$  tensors of rank 1 is unique, up to permutation of the summands. (See also \cite{MR1088949,MR0488592}.)  This result contrasts sharply with the
well-known non-uniqueness of  expressions of matrices of rank at least 2 as sums of rank-1 matrices. The uniqueness of this tensor decomposition is moreover of fundamental interest for a number of applications, ranging from Kruskal's original motivation by latent-class models used in psychometrics, to chemistry and signal processing, as mentioned in \cite{StegSid} and its references. In these fields, the expression of a tensor as a sum of rank-1 tensors is often referred to as  the Candecomp or Parafac decomposition. Recently, Kruskal's theorem has been used as a general tool for investigating the identifiability of a wide variety of statistical models with hidden variables \cite{AMR, ARcov}.

As noted in \cite{StegSid}, Kruskal's original proof was ``rather inaccessible,'' leading a number of authors to work toward a shorter and more intuitive presentation. This thread, which continued to follow the basic outline of Kruskal's approach in which his `Permutation Lemma' plays a key role, culminated in the proof given in \cite{StegSid}. In this paper, we present a new and more concise proof of Kruskal's theorem, Theorem \ref{thm:K} below, that follows an entirely different approach. While the resulting theorem is identical, the alternative argument given here offers a new perspective on the role of Kruskal's explicit condition ensuring uniqueness.

While Kruskal's theorem gives a sufficient condition for uniqueness of a decomposition, the condition is known in general not to be necessary. Of particular note are recent independent works of De Lathauwer \cite{DeLat} and Jiang and Sidiropoulos \cite{JiaSid}, which give a different, though in some ways more narrow, criterion that can ensure uniqueness. See also \cite{StetBerDeLat} for the connection between these works.

It would, of course,  be highly desirable to obtain conditions (more involved than Kruskal's) that would ensure the essential uniqueness of the expression of a rank $r$ tensor as a sum of rank-1 tensors under a wider range of assumptions on the size and rank of the tensor.  Note that both Kruskal's condition and that of \cite{DeLat,JiaSid} can be phrased algebraically, in terms of the non-vanishing of certain polynomials in the variables of a natural parameterization of rank $r$ tensors. This algebraic formulation allows one
to conclude that \emph{generic} rank $r$ tensors of certain sizes have unique decompositions. Having explicit understanding of these polynomial conditions is essential for certain applications, such as in \cite{AMR}. The general problem of determining for which sizes and ranks of generic tensors the decomposition is essentially unique, and what explicit algebraic conditions can ensure uniqueness, remains open. 

\section{Notation}

Throughout, we work over an arbitrary field.

For a matrix such as $M_k$, we use $\mathbf m_j^k$ to denote the $j$th column, $\bar{\mathbf m}_i^k$ to denote the $i$th row, and
$m^k_{ij}$ the $(i,j)$th entry. We use $\langle S\rangle$ to denote the span of a set of vectors $S$. With $[r]=\{1,2,3,\dots,r\}$, we denote by $\mathfrak S_r$ the symmetric group on $[r]$.

Given matrices $M_l$ of size $ s_l\times r$, the matrix triple product $[M_1, M_2,M_3]$ is an $s_1\times s_2\times s_3$ tensor defined as a sum of
$r$ rank-1 tensors by
$$[M_1, M_2,M_3]= \sum_{i=1}^r \mathbf m^1_i\otimes \mathbf m^2_i \otimes \mathbf m^3_i,$$
so
$$[M_1, M_2,M_3](j,k,l)= \sum_{i=1}^r m^1_{ji} m^2_{ki}m^3_{li}.$$

A matrix $A$ of size $t\times s_l$ acts on an $s_1\times s_2\times s_3$ tensor $T$ `in the $l$th coordinate.' For example, with $l=1$
$$(A*_1T)(i,j,k)=\sum_{n=1}^{s_1} a_{in}T(n,j,k),$$
so that $A*_1T$ is of size $t\times s_2\times s_3$.
One then easily checks that
$$A*_1[M_1,M_2,M_3]=[AM_1,M_2,M_3],$$
with similar formulas applying for actions in other coordinates.

\begin{dfn}The \emph{Kruskal rank}, or \emph{K-rank}, of a matrix is the largest number $j$ such that \emph{every} set of $j$ columns is independent.
\end{dfn}

\begin{dfn} We say a triple of matrices $(M_1,M_2,M_3)$ is \emph{of type $(r;a_1,a_2,a_3)$} if each $M_i$ has $r$ columns and the \krank of $M_i$ is at least $r-a_i$. \end{dfn}

In a slight abuse of notation, we will say a product $[M_1,M_2,M_3]$ is of type $(r;a_1,a_2,a_3)$ when the triple $(M_1,M_2,M_3)$ is of that type.

\smallskip

Note that with this definition,  type $(r;a_1,a_2,a_3)$ implies  type $(r;b_1,b_2,b_3)$  as long as $a_i\le b_i$ for each $i$. 
Thus $a_i$ is a bound on the gap between the \krank of the matrix $M_i$ and the number $r$ of its columns. Intuitively, when the $a_i$ are small it should be easier to identify the $M_i$ from the product $[M_1,M_2,M_3]$. 

We will not need to be explicit about the number of rows in any of the $M_i$, though  type $(r;a_1,a_2,a_3)$ of course implies $M_i$ has at least $r-a_i$ rows

\medskip

 \section{The proof}
 
We begin by establishing a lemma that generalizes a basic insight that has been rediscovered many times over the years, in which matrix diagonalizations
arising from matrix slices of a 3-dimensional tensor are used to understand the tensor decomposition.
A few such instances of the appearance of this idea include \cite{MR97k:92011,KasShi}, and other such references are mentioned in \cite{DeLatMV} where the idea is exploited for computational purposes.

\begin{lem} \label{lem:eig} Suppose $(M_1,M_2,M_3)$ is of type $(r; 0,0,r-1)$;  $N_1,N_2,N_3$ are matrices with $r$ columns; and $[M_1,M_2,M_3]=[N_1,N_2,N_3]$. Then there is some permutation
$\sigma\in \mathfrak S_r$ such that the following holds:

Let $ \mathcal I\subseteq[r]$ be any maximal subset (with respect to inclusion) of indices with the property that $\langle \{\mathbf m^3_i\}_{i\in \mathcal I}\rangle$ is 1-dimensional. Then 
\begin{enumerate}
\item\label{enum:lem1}
$\langle \{\mathbf m^j_i\}_{i\in \mathcal I}\rangle =\langle \{\mathbf n^j_{\sigma(i)}\}_{i\in \mathcal I}\rangle$, for $j=1,2,3$
and 
\item \label{enum:lem2}
$\mathcal I$ is also maximal for the property that $\langle \{\mathbf n^3_{\sigma(i)}\}_{i\in \mathcal I}\rangle$ is 1-dimensional.
\end{enumerate}
\end{lem}
\begin{proof} That $(M_1,M_2,M_3)$ is of type $(r; 0,0,r-1)$  means $M_1,M_2$ have full column rank, and $M_3$ has no zero columns.

Choose some vector $\mathbf c$ that is not orthogonal to any of the columns of $M_3$, so that $\mathbf c^TM_3$ has no zero entries. Then
$$A=\mathbf c^T*_3[M_1,M_2,M_3]= [M_1,M_2,\mathbf c^TM_3]=M_1\diag(\mathbf c^TM_3)M_2^T$$
is a matrix of rank $r$. Since
$$ A=\mathbf c^T*_3[N_1,N_2,N_3]= [N_1,N_2,\mathbf c^TN_3]=N_1\diag(\mathbf c^TN_3) N_2^T,$$ 
$N_1$ and $N_2$ must also have rank $r$, and $\mathbf c^TN_3$ has no zero entries. These two expressions for $A$ also show that the span of the columns of $M_j$ is the same as that of the columns of $N_j$ for $j=1,2$. Expressing the columns of $M_j$ and $N_j$ in terms of a basis given by the columns of $M_j$, we may henceforth assume
$M_1=M_2= I_r$, the $r\times r$ identity, and $N_1,N_2$ are invertible. Thus $A=\diag(\mathbf c^TM_3)$.

Now let $S_i$ denote the slice of $[M_1,M_2,M_3]=[N_1,N_2,N_3]$ with fixed third coordinate $i$, so $S_i$ is an $r\times r$ matrix. 
Recalling that $\bar {\mathbf  m}^j_i$ and
$\bar{\mathbf n}^j_i$ denote the $i$th rows of $M_j$ and $N_j$, we have
$$S_i=\diag(\bar {\mathbf m}^3_i)=N_1\diag(\bar {\mathbf n}^3_i)N_2^T.$$
Note the matrices 
$$S_iA^{-1}=\diag(\bar {\mathbf  m}_i^3) \diag(\mathbf c^TM_3)^{-1} =N_1\diag(\bar {\mathbf  n}^3_i)\diag(\mathbf c^TN_3)^{-1} N_1^{-1},$$ for various choices of $i$, commute. Thus  their (right) simultaneous eigenspaces are determined. But from the two expressions for $S_iA^{-1}$ we see its $\alpha$-eigenspace is spanned by the set
\begin{equation*}\{\mathbf e_j=\mathbf m_j^1 ~|~m^3_{i,j}/ (\mathbf c^T\mathbf m^3_j) =\alpha\},\label{eq:m3}\end{equation*} and also by  the set
\begin{equation*}\{\mathbf n^1_j~|~ n^3_{i,j}/(\mathbf c^T\mathbf n^3_j) = \alpha\}.\label{eq:n3}\end{equation*}

A simultaneous eigenspace for the $S_iA^{-1}$ is thus spanned by the set $\{\mathbf e_j\}_{j\in \mathcal I}$ where $\mathcal I$ is a maximal set of indices with the property that if $j,k\in \mathcal I$, then 
\begin{equation}m_{i,j}^3/(\mathbf c^T\mathbf m^3_j) =m_{i,k}^3/(\mathbf c^T\mathbf m^3_k), \text{ for all $i$}.\notag\end{equation}
This condition is equivalent to  $ \mathbf m_j^3$ and $\mathbf m_k^3$ being scalar multiples of one another. Such a set $\mathcal I$ is therefore exactly of the sort described in the statement of the lemma. As the simultaneous eigenspaces are also spanned by similar sets defined in terms of the columns of $N_1$,  one may choose a permutation $\sigma$ so that claim \ref{enum:lem2} holds,  as well as  claim \ref{enum:lem1} for $j=1$. 

The case $j=2$ of claim 1 is similarly proved using the transposes of $A$ and the $S_i$. As the needed permutation of the columns of the $N_j$ in the two cases of $j=1,2$
is dependent only on the maximal sets $\mathcal I$,  a common $\sigma$ may be chosen. Finally, the case $j=3$ follows from  equating eigenvalues in the two expressions giving diagonalizations for $S_iA^{-1}$, to see that for all $i$
$$m^3_{i,j}/\mathbf c^T\mathbf m^3_j=n^3_{i,\sigma(j)}/\mathbf c^T \mathbf n^3_{\sigma(j)},$$
so $\mathbf m^3_j$ and $\mathbf n^3_{\sigma(j)}$ are scalar multiples of one another.
\end{proof}

This lemma quickly yields a special case of Kruskal's theorem, when two of the matrices in the product are asumed to have full column rank.

\begin{cor} \label{cor:eig} Suppose $(M_1,M_2,M_3)$ is of type $(r;0,0,r-2)$;  $N_1,N_2,N_3$ are matrices with $r$ columns; and $[M_1,M_2,M_3]=[N_1,N_2, N_3]$.
Then there exists some permutation matrix $P$ and invertible diagonal matrices $D_i$ with $D_1D_2D_3=I_r$ such that $N_i=M_iD_iP$.
\end{cor}
\begin{proof}
Since $(M_1,M_2,M_3)$ is also of type $(r;0,0,r-1)$, we may apply Lemma \ref{lem:eig}. 
As in the proof of that lemma, we may also assume $M_1=M_2=I_r$.
But $M_3$ has \krank at least 2, so every pair of columns is independent. Thus the maximal sets of indices in Lemma \ref{lem:eig} are all singletons. Thus with $P$ acting to permute columns by $\sigma$, the one-dimensionality of all eigenspaces shows there is a permutation $P$ and invertible diagonal matrices $D_1,D_2$ with $N_i=M_iD_iP=D_iP$ for $j=1,2$.

Thus $[M_1,M_2,M_3]=[N_1,N_2,N_3]$ implies $$[I_r,I_r, M_3]=[D_1P,D_2P, N_3]=[D_1,D_2,N_3P^T]=[I_r,I_r,N_3 P^T D_1D_2],$$ which shows
$M_3=N_3P^TD_1D_2 $. Setting $D_3=(D_1D_2)^{-1}$, we find $N_3=M_3D_3P$.
\end{proof}

We now use the lemma to give a new proof of  Kruskal's Theorem in its full generality. Note that the condition on the $a_i$ stated in the following theorem is equivalent to Kruskal's condition in \cite{MR0444690} that $(r-a_1)+(r-a_2)+(r-a_3)\ge 2r+2$.

\begin{thm} [Kruskal, \cite{MR0444690}]\label{thm:K}
Suppose $(M_1,M_2,M_3)$ is of type $(r;a_1,a_2,a_3)$ with  $a_1+a_2+a_3\le r-2$; $N_1, N_2,N_3$ are matrices with $r$ columns, and $[M_1,M_2,M_3]=[N_1,N_2, N_3]$. Then there exists some permutation matrix $P$ and invertible diagonal matrices $D_i$ with $D_1D_2D_3=I_r$ such that $N_i=M_iD_iP$.
\end{thm}

\begin{proof} We need only consider $a_1+a_2+a_3=r-2$. 
We  proceed by induction on $r$, with the case $r=2$ (and $3$) already established by Corollary \ref{cor:eig}. We may also assume $a_1\le a_2\le a_3$,
We may furthermore restrict to $a_2\ge 1$, since the case $a_1=a_2=0$ is established by Corollary \ref{cor:eig}.

\ 

We first claim that it will be enough to show that, for some $1\le i\le 3$, there is some set of indices $\mathcal J\subset [r]$, $1\le |\mathcal J |\le r-a_i-2$, and a permutation $\sigma\in \mathfrak S_r$ such that
\begin{equation}\langle \{\mathbf m^i_j\}_{j\in \mathcal J}\rangle=\langle \{\mathbf n^i_{\sigma(j)} \}_{j\in \mathcal J}\rangle.
\label{eq:spanmn}
\end{equation}

To see this, if there is such a set $\mathcal J$,  assume for convenience $i=1$ (the cases $i=2,3$ are similar), and the columns of $M_i,N_i$ have been reordered so that $\sigma=id$ and
$\mathcal J=[s].$
Let  $\Pi$ be a matrix with nullspace the span described in equation \eqref{eq:spanmn}. Then
$$[\Pi M_1,M_2,M_3]=\Pi*_1[M_1,M_2,M_3]=\Pi*_1[ N_1,N_2,N_3]=[\Pi N_1,N_2,N_3].$$
But since the first $s$ columns of $\Pi M_1$ and $\Pi N_1$ are zero, these triple products can be expressed as
triple products of matrices with only $r-s$ columns. That is, using the symbol `$\,\widetilde \  \,$' to denote deletion of the first $s$ columns,
$$[\Pi \widetilde M_1,\widetilde M_2,\widetilde M_3]=[\Pi \widetilde  N_1,\widetilde N_2,\widetilde N_3].$$

For $i=2,3$, since $M_i$ has \krank $\ge r-a_i$, the matrix $\widetilde M_i$ has \krank $\ge \min (r-a_i,r-s)$. Since the nullspace of $\Pi$ is spanned by the first $s$ columns of $M_1$, and $M_1$ has \krank $\ge r-a_1$, ones sees that $\Pi \widetilde M_1$ has \krank $\ge r-s- a_1$, as follows:
For any set of $r-s-a_1$ columns of $\Pi \widetilde M_1$, consider the corresponding columns of $M_1$, together with the first $s$ columns. This set of $r-a_1$ columns of $M_1$ is therefore independent, so the span of its image under $\Pi$ is of dimension $r-s-a_1$. This span must then have as a basis the chosen set of $r-s-a_1$ columns of $\Pi \widetilde M_1$, which are therefore independent.
Thus $[\Pi \widetilde M_1,\widetilde M_2,\widetilde M_3]$ is of type $(r-s;a_1,b_2,b_3)$, where $b_i=\max(0, a_i-s)$ for $i=2,3$.  Note also that $s\le r-a_1-2$ implies
$a_1+b_2+b_3\le r-s-2$.

We may thus apply the inductive hypothesis to $[\Pi \widetilde M_1,\widetilde M_2,\widetilde M_3]=[\Pi \widetilde  N_1,\widetilde N_2,\widetilde N_3],$
and, after an allowed permutation and scalar multiplication of the columns of the $N_i$,  conclude that $\widetilde M_i=\widetilde N_i$ for $i=2,3$.
But this means we can now take the set $\mathcal J$ described in equation \eqref{eq:spanmn} to be a singleton set $\{j\}$, with $j>s$, and $i=2$.  Again applying the argument developed thus far implies that, allowing for a possible permutation and rescaling, all but the $j$th columns of $M_3$ and $N_3$ are identical. As $\mathbf m^3_j=\mathbf n^3_j$, this shows $M_3=N_3$.  Applying this argument yet again, with  $i=3$, and varying  choices of $j$, then shows $M_1=N_1$ and $M_2=N_2$, up to the allowed permutation and rescaling. The claim is thus established.

\medskip

We next argue that some set of columns of some $M_i$, $N_i$ meets the hypotheses of the above claim.

Let $\Pi_3$ be any matrix with nullspace $\langle \{ \mathbf n^3_i\}_{1\le i\le a_1+a_2}\rangle$, spanned by the first $a_1+a_2$ columns of $N_3$. 
 Let $\mathcal Z$ be the set of indices of all zero columns of $\Pi_3M_3$.
Since every set of $r-a_3=a_1+a_2+2$ columns of $M_3$ is independent, $| \mathcal Z|\le a_1+a_2$.
Note also that at least 2 columns of $\Pi_3M_3$ are independent, since the span of any $a_1+a_2+2$ columns of $\Pi_3 M_3$ is at least 2 dimensional. 

Let $\mathcal S_1,\mathcal S_2$ be any disjoint subsets of $[r]$ such that $|\mathcal S_1|=a_2$, $|\mathcal S_2|=a_1$, $\mathcal Z\subseteq \mathcal S_1\cup \mathcal S_2=\mathcal S$, and $\mathcal S$ excludes at least two indices of independent columns of $\Pi_3M_3$. Let $\Pi_1=\Pi_1(\mathcal S_1) $ be any matrix with nullspace $\langle \{ \mathbf m^1_i\}_{i\in \mathcal S_ 1}\rangle$, and  let $\Pi_2=\Pi_2(\mathcal S_2)$ be any matrix with nullspace $\langle\mathbf  \{\mathbf m^2_i\}_{i\in \mathcal S_2}\rangle$.

Now consider 
\begin{multline*}[\Pi_1M_1,\Pi_2 M_2,\Pi_3 M_3] =\Pi_3*_3(\Pi_2*_2(\Pi_1*_1[M_1,M_2,M_3]))\\=\Pi_3*_3(\Pi_2*_2(\Pi_1*_1[N_1,N_2,N_3])))=[\Pi_1N_1,\Pi_2N_2,\Pi_3N_3].\end{multline*}
By the specification of the nullspace of $\Pi_3$, the columns of all $N_i$ with indices in $[a_1+a_2]$ can be deleted in this last product. In the first product, one can similarly delete the columns of the $M_i$ with indices in $\mathcal S$, due to the specifications of the nullspaces of $\Pi_1$ and $\Pi_2$. Using `$\,\widetilde\ \,$' to denote the deletion of these columns, we have
\begin{equation}[\Pi_1\widetilde M_1,\Pi_2 \widetilde M_2,\Pi_3 \widetilde M_3] =[\Pi_1\widetilde N_1,\Pi_2\widetilde N_2,\Pi_3\widetilde N_3],
\label{eq:Pprod}\end{equation}
where these products involve matrix factors with $r-a_1-a_2=a_3+2$ columns. 

The matrix $\Pi_1\widetilde M_1$ in fact has full column rank. To see this, note that it can also be obtained from
$M_1$ by (a) first deleting columns with indices in $\mathcal S_2$, then (b) multiplying on the left by $\Pi_1$, and finally (c) deleting the columns arising from those in $M_1$ with indices in $\mathcal S_1$. Since  $M_1$ has \krank at least $r-a_1$, step (a) produces a matrix with $r-a_1$ columns, and  full column rank. Since the nullspace of $\Pi_1$ is spanned by certain of the columns of this matrix,  step (b) produces a matrix whose non-zero columns are independent. Step (c) then deletes all zero columns to give a matrix of full column rank.
Similarly, the matrix $\Pi_2\widetilde M_2$ has full column rank.

Noting that $\Pi_3\widetilde M_3$ has no zero columns since $\mathcal Z\subseteq\mathcal S$, we may thus apply Lemma \ref{lem:eig} to the products of equation \eqref{eq:Pprod}. In particular, we find that there is some $\sigma\in \mathfrak S_r$ with $\sigma([r]\smallsetminus \mathcal S)=[r]\smallsetminus[a_1+a_2]$ such that  if $\mathcal I$ is a maximal subset of $[r]\smallsetminus\mathcal S$ with respect to the property that $\langle \{\Pi_3 \mathbf m_i^3 \}_{i\in I}\rangle$ is 1-dimensional, then 
\begin{equation}
\langle \{ \Pi_j \mathbf m^j_i\}_{i\in \mathcal I }\rangle =\langle \{ \Pi_j\mathbf n^j_{\sigma(i)}\}_{i\in \mathcal I}\rangle
\label{eq:Isig}\end{equation} for $j=1,2,3$.

Since we chose $\mathcal S$ to exclude indices of two independent columns of $\Pi_3M_3$, 
there will be such a maximal subset $\mathcal I$ of $[r]\smallsetminus \mathcal S$  that contains at most half the indices. We thus pick such an $\mathcal I$ with
 $|\mathcal I|\le \lfloor (r-a_1-a_2)/2 \rfloor=\lfloor a_3/2\rfloor +1$, and 
consider two cases:

\textbf{Case $a_1=0$:} 
Then $\mathcal S_2=\emptyset$, and $\Pi_2$ has trivial nullspace and thus may be taken to be the identity.
Since $a_3\ge a_2\ge 1$, this implies $|\mathcal I|\le a_3=r-a_2-2$. The sets
$ \{  \mathbf m^2_i\}_{i\in \mathcal I }$ and $ \{ \mathbf n^2_{\sigma(i)}\}_{i\in \mathcal I}$ therefore satisfy the hypotheses of the claim.

\textbf{Case $a_1\ge 1$:}
Note that  $|\mathcal I|+a_2 +1\le \lfloor a_3/2 \rfloor +a_2+2 < a_2+a_3+2=r-a_1$, so for any index $k$, the columns of $M_1$ indexed by $\mathcal I\cup \mathcal S_1\cup \{k\}$ are independent.
This then implies that for $j=1$ the spanning set on the left of equation \eqref{eq:Isig} is independent, so the spanning set on the right is as well. Thus
the set $\{ \mathbf n^1_{\sigma(i)}\}_{i\in \mathcal I}$ is also independent.
Note next that equation \eqref{eq:Isig}
implies that, for $i\in \mathcal I$, there are scalars $b_j^i,c_k^i$ such that
\begin{equation}\mathbf n^1_{\sigma(i)}-\sum_{j\in \mathcal I} b_j^i \mathbf m_j^1=\sum_{k\in \mathcal S_1} c_k^i \mathbf m^1_k.
\label{eq:n1}\end{equation}
Now for any $p\in \mathcal S_1$, $q\in \mathcal S_2$, let 
$$\mathcal S_1'=(\mathcal S_1\smallsetminus \{p\})\cup \{q\},\ \ \mathcal S_2'=(\mathcal S_2\smallsetminus \{q\})\cup \{p\} .$$
Choosing $\Pi_1'$ and $\Pi_2'$ to have nullspaces determined as above by the index sets $\mathcal S_1'$ and  $\mathcal S_2'$, and applying  Lemma \ref{lem:eig} to $[\Pi_1'M_1,\Pi_2'M_2,\Pi_3M_3]= [\Pi_1'N_1,\Pi_2'N_2,\Pi_3N_3]$, similarly shows that for some permutation $\sigma'$ and any $i'\in \mathcal I$ there are scalars $d_k^{i'},f_k^j$ such that
\begin{equation}\mathbf n^1_{\sigma'(i')}-\sum_{j\in \mathcal I} d_j^{i'} \mathbf m_j^1=\sum_{l\in \mathcal S_1'}f_l^{i'} \mathbf m^1_l.
\label{eq:n2}\end{equation} Note that since the same $\Pi_3$ was used, the set $\mathcal I$ is unchanged here, and  $\sigma$ and $\sigma'$ must have the same image on $\mathcal I$.   Picking $i'\in \mathcal I$ so that $\sigma'(i')=\sigma(i)$, and
subtracting equation \eqref{eq:n1} from \eqref{eq:n2}  shows
$$\sum_{j\in \mathcal I} (b_j^i-d_j^{i'}) \mathbf m_j^1= \sum_{k\in \mathcal S_1\smallsetminus\{p\} } (f_k^{i'}-c_k^i) \mathbf m^1_k   +   f_q^{i'} \mathbf m^1_q-c_p^i \mathbf m^1_p.$$
But since the columns of $M_1$ appearing in this equation are independent, we see that $f_q^{i'}=c_p^i=0$. By varying $p$,
we conclude
that $ \mathbf n^1_{\sigma(i)} \in  \langle \{ \mathbf m_i^1 \}_{i\in \mathcal I}\rangle$. 
Thus  $ \langle\{ \mathbf n^1_{\sigma(i)}\}_{i\in \mathcal I}\rangle \subseteq  \langle \{ \mathbf m_i^1 \}_{i\in \mathcal I}\rangle$. 
Since both of these spanning sets are independent, and of the same cardinality, their spans must be equal.
Since $|\mathcal I|\le r-a_1-2$, the set $\mathcal I$ satisfies the hypotheses of the claim.

\end{proof}

\bibliographystyle{amsplain}
\bibliography{phylo}

\end{document}